\documentclass{amsart}
\usepackage{amsmath, amscd, amssymb, amsthm}
\usepackage{bbm}
\usepackage{latexsym}
\usepackage{amsfonts}
\usepackage{graphicx}
\usepackage[all,cmtip]{xy}

\usepackage[
colorlinks=true, linkcolor=black,breaklinks=true,
urlcolor=blue, citecolor=black]{hyperref}%
\usepackage{geometry}
\newtheorem{lemma}{Lemma}

\newcommand{\be}{\begin{equation}}
\newcommand{\ee}{\end{equation}}
\newcommand{\bea}{\begin{eqnarray}}
\newcommand{\eea}{\end{eqnarray}}

\def\XXint#1#2#3{{\setbox0=\hbox{$#1{#2#3}{\int}$ }
\vcenter{\hbox{$#2#3$ }}\kern-.6\wd0}}

\begin{document}

\title[Threefold with zero real bisectional curvature]{Hermitian threefolds with vanishing real bisectional curvature}


\author{Wu Zhou}
\address{Wu Zhou. School of Mathematics, Southwest Minzu University, 16, South Section, 1st Ring Road, WuHou District, Chengdu, Sichuan 610041, China}
\email{{465840124@qq.com}}

\author{Fangyang Zheng} \thanks{The research is partially supported by NSFC grant \# 12071050 and Chongqing Normal University.}
\address{Fangyang Zheng. School of Mathematical Sciences, Chongqing Normal University, Chongqing 401331, China}
\email{{20190045@cqnu.edu.cn}}

\subjclass[2010]{53C55 (primary), 53C05 (secondary)}
\keywords{Hermitian manifold; Chern connection; holomorphic sectional curvature; real bisectional curvature}

\begin{abstract}
We examine the class of compact Hermitian manifolds with constant holomorphic sectional curvature. Such manifolds are conjectured to be K\"ahler (hence a complex space form) when the constant is non-zero and Chern flat (hence a quotient of a complex Lie group) when the constant is zero. The conjecture is known in complex dimension two but open in higher dimensions. In this paper, we establish a partial solution in complex dimension three by proving that any compact Hermitian threefold with zero real bisectional curvature must be Chern flat. Real bisectional curvature is a curvature notion introduced by Xiaokui Yang and the second named author in 2019, generalizing holomorphic sectional curvature. It is equivalent to the latter in the K\"ahler case and is slightly stronger in general. \end{abstract}

\maketitle


\markleft{Zhou and Zheng}
\markright{Threefold with zero real bisectional curvature}

\section{Introduction and statement of result}

A long-standing problem in complex differential geometry is the classification of compact Hermitian manifolds with constant holomorphic sectional curvature. For K\"ahler metrics,  it is well known that holomorphic sectional curvature determines the entire curvature tensor, and complete K\"ahler manifolds with constant holomorphic sectional curvature are complex space forms, namely, quotients of either ${\mathbb C}{\mathbb P}^n$, or ${\mathbb C}^n$, or the complex hyperbolic space ${\mathbb C}{\mathbb H}^n$, equipped with (scaling of) the standard metrics.

When the metric is non-K\"ahler, however, the curvature tensor no longer obeys all the symmetries as in the K\"ahler case, thus the holomorphic sectional curvature does not determine the whole curvature tensor in general. Nevertheless, the following conjecture is widely believed to be true in the field:

\vspace{0.3cm}

\noindent {\bf Conjecture 1.} {\em Suppose $(M^n,g)$ is a compact Hermitian manifold whose holomorphic sectional curvature $H$ is equal to a constant $c$. Then $g$ is K\"ahler when $c\neq 0$ and $g$ is Chern flat when $c=0$.}

\vspace{0.3cm}

The conjecture is known to be true when $n=2$, by Balas and Gauduchon \cite{BG} in 1985 for the  $c\leq 0$ case (see also \cite{Gauduchon} and \cite{Balas} for earlier work), and by Apostolov, Davidov and Muskarov \cite{ADM} in 1996 for the $c>0$ case. For $n\geq 3$, the conjecture is still wide open, though some partial results are known. For instance, Kai Tang in \cite{Tang} proved the conjecture under the additional assumption that the metric is Chern K\"ahler-like, meaning that the curvature tensor of the Chern connection obeys all the K\"ahler symmetries. An interesting new development was obtained by Chen-Chen-Nie \cite{CCN} in which they established the conjecture for the case $c\leq 0$ under the additional assumption that $g$ is locally conformally K\"ahler. They also pointed out the necessity of the compactness assumption in the conjecture by explicit examples.

Note that compact Chern flat manifolds were classified by Boothby \cite{Boothby} in 1958. They are compact quotients of complex Lie groups equipped with left invariant metrics. When $n\geq 3$, such manifolds can be non-K\"ahler.

For the conjecture, an important special case is when the manifold  is {\em balanced,} meaning that the K\"ahler form $\omega$ of $g$ is co-closed (namely,  $d (\omega^{n-1})=0$ where $n$ is the dimension). This case already contains a lot of the interesting non-K\"ahler manifolds, even in complex dimension three, as it includes all the twistor spaces for instance. So far the conjecture is still largely unknown even for this special case, and the main purpose of this article is establish a partial answer to it, by assuming that the Hermitian threefold $(M^3,g)$ has vanishing real bisectional curvature, thus providing a partial supporting evidence to the conjecture.

Let $(M^n,g)$ be a compact Hermitian manifold. Denote by $R$ the curvature tensor of the Chern connection. The holomorphic sectional curvature $H$ is defined by
$$ H(X) = \frac{R_{X\overline{X}X\overline{X}} } {|X|^4}  $$
where $X$ is any non-zero complex tangent vector of $(1,0)$ type. In \cite{YZ}, Xiaokui Yang and the second named author introduced the notion of real bisectional curvature $B$ for Hermitian metric in the following way. For $p\in M$, let $e= \{ e_1, \ldots , e_n\}$ be a unitary tangent frame at $p$, and let $a=\{ a_1, \ldots , a_n\} $ be non-negative constants with $|a|^2 = a_1^2 + \cdots + a_n^2 > 0$. Define the {\em real bisectional curvature} of $g$ by
$$ B_g(e,a)=\frac{1}{|a|^2} \sum_{i, j=1}^n R_{i\overline{i}j\overline{j}} a_ia_j.$$
We say that $B>0$ at $p$ if $ B_g(e,a)>0$ for any choice of unitary tangent frame $e$ at $p$ and any choice of $a$ as above. One can define $B\geq 0, \, =0, \, < 0$, or $\leq 0$ similarly. The notion $B$ is a natural extension of $H$ from the point of view of Berger's Lemma. When the metric $g$ is K\"ahler, $B>0  \Longleftrightarrow H>0 $. Similarly, $B=(<,\, \leq ,\, \geq )\,0 \Longleftrightarrow H=( <, \,\leq , \,\geq )\,0$. In other words, $B$ and $H$ have the same signs.

When the metric $g$ is non-K\"ahler, $B$ is a  stronger notion than $H$ in general, namely, $B>0$ implies $H>0$ but not the other way around. There are local examples constructed in \cite{YZ} which illustrate the difference, although global examples are difficult to construct in general.

A Hermitian manifold $(M^n,g)$ is said to have constant real bisectional curvature, denoted by $B=c$ where $c$ is a constant, if $B_g(e,a)=c|a|^2$ for any $p\in M$ and any choice of $e$ and $a$ as above. The following conjecture was raised in \cite[Conjecture 1.5]{YZ} (note that there was a misprint there):

\vspace{0.3cm}

\noindent {\bf Conjecture 2 (\cite{YZ}).} {\em Suppose $(M^n,g)$ is a compact Hermitian manifold with constant real bisectional curvature $B=c$.  Then $c$ must be $0$ and  $g$ is Chern flat.}

\vspace{0.3cm}

The result \cite[Theorem 1.4]{YZ} states that $c>0$ cannot occur, and if $c=0$, then $g$ is balanced, and its Chern curvature tensor has  vanishing first, second, third Ricci, and satisfies
\begin{equation} \label{eq:1.1}
R_{X\overline{Y}Z\overline{W}} + R_{Z\overline{W}X\overline{Y}}=0
\end{equation}
for any type $(1,0)$ tangent vectors $X$, $Y$, $Z$, $W$. So for Conjecture 2 with $n\geq 3$, one still needs to show that $c<0$ cannot occur, and when $c=0$, the Chern curvature tensor $R$ must be zero.

The main result of this article is to prove the following:

\vspace{0.3cm}

\noindent {\bf Theorem 1.} {\em Let $(M^3,g)$ be a compact Hermitian threefold with vanishing real bisectional curvature. Then  $g$ must be Chern flat.}

\vspace{0.3cm}

In other words, we will show that if a compact Hermitian manifold $(M^3,g)$ satisfies the curvature condition (\ref{eq:1.1}), then it must be Chern flat. The proof is based on utilizing the first and second Bianchi identities for this special kind of Hermitian manifolds, Boothby's Bochner lemma computation, and an observation about a particular form of unitary tangent frame on balanced threefolds.

The last point is a technical result for balanced threefold, which is interesting in its own right and might have other applications. So we state it as a proposition:

 \vspace{0.3cm}

\noindent {\bf Proposition 2.} {\em Let $(M^3,g)$ be a balanced Hermitian threefold. Then for any $p\in M^3$, there exists a unitary frame of type $(1,0)$ tangent vectors in a neighborhood of $p$ such that $T^1_{1k}=T^2_{2k}=T^3_{3k}=0$ for any $1\leq k\leq 3$. Here $T^j_{ik}$ are the components under the frame of the torsion tensor of the Chern connection.}

\vspace{0.3cm}

The article is organized as follows. In the next section, we will collect some known results from existing literature, set up the notations and prove some preliminary results. In the third section we will prove Proposition 2 and Theorem 1.

\vspace{0.3cm}

\section{Preliminaries}

Let $(M^n,g)$ be a Hermitian manifold. Denote by $J$ the almost complex structure and by $g=\langle \, , \, \rangle$ the metric, extended bilinearly over ${\mathbb C}$. Let $\nabla$ be the Chern connection. Its torsion and curvature are denoted by $T$ and $R$, respectively:
$$ T(x,y)=\nabla_xy - \nabla_yx - [x,y], \ \ \ R(x,y,z,w) = \langle R_{xy}z,\  w\rangle= \langle \nabla_x\nabla_yz -\nabla_x\nabla_yz -\nabla_{[x,y]}\,z,\  w\rangle $$
for any  tangent vector fields $x$, $y$, $z$, $w$ on $M$. We will extend $T$ and $R$ linearly over ${\mathbb C}$, and still denote them by the same letters. From now on, we will use $X$, $Y$, $Z$, $W$ to denote complex tangent vectors of type $(1,0)$, namely, $X=x-\sqrt{-1}Jx\,$ for some real tangent vector $x$. It is well known that the torsion and curvature of the Chern connection enjoy the following properties:
\begin{equation} \label{eq:1.2}
T(X, \overline{Y})=0, \ \ \ R(X,Y,\ast , \ast) = R(\ast , \ast , Z, W)=0.
\end{equation}
So the only possibly non-zero components of $R$ are $R(X, \overline{Y}, Z, \overline{W})$, which we will denote by $R_{X \overline{Y} Z \overline{W}}$ for convenience. Let $e=\{ e_1, \ldots , e_n\}$ be a local unitary frame of type $(1,0)$ tangent vectors. Following \cite{YangZ}, we will denote the components of $T$ by
\begin{equation} \label{eq:1.3}
T(e_i, e_k)= \sum_{j=1}^n T_{ik}^j e_j.
\end{equation}
Note that our $T^j_{ik}$ here is equal to twice of that in \cite{YangZ}, where the peculiar notation was for the convenience of some formulas in \cite{YangZ}. The first Bianchi identity leads to the following (see for instance \cite[Lemma 7]{YangZ}):
\begin{equation} \label{eq:B1}
R_{k\overline{j}i\overline{\ell}} - R_{i\overline{j}k\overline{\ell}} =  T^{\ell}_{ik,\,\overline{j}}
\end{equation}
 for any $1\leq i,j,k,\ell \leq n$, where the index after comma stands for covariant derivative under $\nabla$. Also, by the second Bianchi identity, which for any connection is in the form
 $$ {\mathfrak S}\{ (\nabla_xR)_{yz} + R_{T(x,y)\, z} \} = 0 ,$$
 where ${\mathfrak S}$ means the sum over all cyclic permutation of $x,y,z$, and when applied to the special case of Chern connection $\nabla$, we get
 \begin{equation} \label{eq:B2}
R_{k\overline{j}i\overline{\ell},\, m} - R_{m\overline{j}k\overline{\ell}, \,i} \,= \, \sum_{r=1}^n T^{r}_{im} R_{r\overline{j}i\overline{\ell}}
\end{equation}
 for any $1\leq i,j,k,\ell , m\leq n$, where the indices after comma again stand for covariant derivatives under $\nabla$. For our later proofs, we will also need the following two commutation formulas:
 \begin{eqnarray}
 T^k_{j\ell,\,\overline{i}\overline{m}} - T^k_{j\ell,\,\overline{m}\overline{i}} & = &  \sum_{r=1}^n \overline{T^r_{im }}\, T^k_{j\ell , \,\overline{r}} \label{eq:C1}  \\
 T^j_{ik,\,\overline{\ell}m} - T^j_{ik,\,m\overline{\ell }} & = &  \sum_{r=1}^n \big(T^r_{ik}R_{m\overline{\ell} r\overline{j}} - T^j_{rk}R_{m\overline{\ell} i\overline{r}} - T^j_{ir}R_{m\overline{\ell} k\overline{r}} \big) \label{eq:C2}
 \end{eqnarray}
for any indices $i$, $j$, $k$, $\ell$, $m$. At any fixed point $p\in M$, we can choose a local unitary frame $e$ so that the connection matrix $\theta$ for $\nabla$ vanishes at $p$. Using this frame, a straight forward computation will lead to the above two formulas, so we omit the computations here.

Next, recall that the first, second and third Ricci curvature tensors of the Chern connection are defined by
$$
\rho^{(1)}_{i\overline{j}} = \sum_{r=1}^n R_{i\overline{j}r\overline{r}}, \ \ \ \ \rho^{(2)}_{i\overline{j}} = \sum_{r=1}^n R_{r\overline{r} i\overline{j}}, \ \ \ \ \rho^{(3)}_{i\overline{j}} = \sum_{r=1}^n R_{r\overline{j}i\overline{r}}
$$
under any unitary frame  $e$. Letting $m=\ell$ and sum over in (\ref{eq:C2}), we get
\begin{equation}
\label{eq:C3}
\sum_{\ell =1}^n \big( T^j_{ik,\overline{\ell}\ell } - T^j_{ik,\ell \overline{\ell }}\big)  =  \sum_{r=1}^n \big(T^r_{ik}\rho^{(2)}_{ r\overline{j}} - T^j_{rk}\rho^{(2)}_{ i\overline{r}} - T^j_{ir}\rho^{(2)}_{ k\overline{r}} \big)
\end{equation}

Let us denote by $\{ \varphi_1, \ldots , \varphi_n\}$ the unitary coframe of local $(1,0)$-forms dual to $e$. Gauduchon's torsion $1$-form $\eta$ (see \cite{Gauduchon1}) is the global $(1,0)$-form on $M^n$ defined by the equation
\begin{equation} \label{eq:eta}
\partial (\omega ^{n-1}) = - \eta \wedge \omega^{n-1}
\end{equation}
where $\omega$ is the K\"ahler form of the Hermitian metric $g$. Under the local frame $e$ and coframe $\varphi$, we have \begin{equation}
\eta = \sum_{i=1}^n \eta_i \varphi_i, \ \ \ \ \ \ \mbox{where} \ \ \ \
 \eta_i = \sum_{r=1}^n T^r_{ri}.
 \end{equation}
Again note that our $\eta$ is equal to twice of $\eta$ in \cite{YangZ}, just like $T^j_{ik}$.

\vspace{0.3cm}

\section{Proof of Theorem 1}

In this section, we will prove Theorem 1 stated in the first section. Let $(M^n,g)$ be a Hermitian manifold with real bisectional curvature $B$ equal to a constant $c$, namely, for any unitary frame $e$ and any nonnegative constants $a=\{ a_1, \ldots , a_n\} $ with $|a|^2=|a_1|^2 + \cdots +|a_n|^2 > 0$, it holds that $B(e,a)= \sum_{i,j=1}^n R_{i\overline{i}j\overline{j}}a_ia_j  = c|a|^2$. In this case, it was proved in \cite[(3.1)]{YZ} that the Chern curvature tensor $R$ satisfies
\begin{equation} \label{eq:3.1}
R_{i\overline{j}k\overline{\ell}} + R_{k\overline{\ell}i\overline{j}} =2c \,\delta_{i\ell}\delta_{kj}
\end{equation}
for any indices $1\leq i,j,k,\ell\leq n$ under a unitary frame $e$. From this, we will deduce the following

\vspace{0.3cm}

\begin{lemma} Suppose  $(M^n,g)$ is a Hermitian manifold with real bisectional curvature $B$ equals to a constant $c$. Then under any unitary local frame $e$, the components of the torsion and curvature satisfy the following
\begin{eqnarray}
 T^{\ell}_{ik, \overline{j}} - T^{j}_{ik, \overline{\ell}} & = &  \, 2c \,( \delta_{ij} \delta_{k\ell} - \delta_{i\ell } \delta_{kj})   \label{eq:3.2} \\
 2 R_{i\overline{j}k\overline{\ell}} \ \ & = & \, 2c \,\delta_{ij} \delta_{k\ell } - T^{\ell }_{ik, \overline{j}} - \overline{T^k_{j\ell , \overline{i} } } \label{eq:3.3}
\end{eqnarray}
where indices after comma stands for covariant derivatives with respect to the Chern connection $\nabla$.
\end{lemma}

\begin{proof}
We start from the formula (\ref{eq:B1}). Interchange $j$ and $\ell$ in (\ref{eq:B1}), we get
\begin{equation} \label{eq:B1a}
R_{k\overline{\ell}i\overline{j}} - R_{i\overline{\ell}k\overline{j}} =  T^{j}_{ik,\,\overline{\ell}}
\end{equation}
Subtract  (\ref{eq:B1a}) from (\ref{eq:B1}), and apply (\ref{eq:3.1}), we get
\begin{eqnarray*}
T^{\ell}_{ik,\,\overline{j}} - T^{j}_{ik,\,\overline{\ell}} & = &  R_{k\overline{j}i\overline{\ell}} -  R_{i\overline{j}k\overline{\ell}} -  R_{k\overline{\ell}k\overline{j}} +  R_{i\overline{\ell}k\overline{j}}\\
& = & (R_{k\overline{j}i\overline{\ell}}+ R_{i\overline{\ell}k\overline{j}} ) - (R_{i\overline{j}k\overline{\ell}} +  R_{k\overline{\ell}k\overline{j}})\\
& = & 2c\,(\delta_{ij} \delta_{k\ell} - \delta_{i\ell} \delta_{kj})
\end{eqnarray*}
 This establishes (\ref{eq:3.2}). Now if we interchange $i$ and $j$ and interchange $k$ and $\ell$ in (\ref{eq:B1}) and then take complex  conjugation on both sides, we get
 \begin{equation} \label{eq:B1b}
R_{i\overline{\ell}k\overline{j}} - R_{i\overline{j}k\overline{\ell}} =  \overline{T^{k}_{j\ell ,\,\overline{i}}}
\end{equation}
Taking the difference between (\ref{eq:B1a}) and (\ref{eq:B1b}), we get
\begin{eqnarray*}
T^{j}_{ik,\,\overline{\ell}} - \overline{T^{k}_{j\ell ,\,\overline{i}}} & = &  (R_{i\overline{j}k\overline{\ell}} +R_{k\overline{j}i\overline{\ell}} ) -  2 R_{i\overline{\ell}k\overline{j}} \\
& = & 2c \,\delta_{i\ell}\delta_{kj} - 2 R_{i\overline{\ell}k\overline{j}}
\end{eqnarray*}
This gives us the expression of $R_{i\overline{\ell}k\overline{j}}$. Interchange $j$ and $\ell$, we get (\ref{eq:3.3}).
\end{proof}

\vspace{0.3cm}

\begin{lemma} Suppose  $(M^n,g)$ is a compact Hermitian manifold with vanishing real bisectional curvature. Then under any unitary local frame $e$, it holds that
\begin{equation}
 \sum_{s=1}^n T^{j}_{is, \overline{s}} =0, \ \ \ \ \ \sum_{s=1}^n  T^{j}_{ik, \overline{s}s} =  \sum_{r,s=1}^n  T^r_{is}T^j_{rk,\overline{s}}.  \label{eq:3.6}
\end{equation}
for any $i$, $j$, $k$.
\end{lemma}

\begin{proof}
When $c=0$, we know from \cite[Theorem 1.4]{YZ} that $\eta =0$. So in (\ref{eq:3.2}), let $k=\ell$ and sum from $1$ to $n$, we get
$$ \sum_{s=1}^n T^{j}_{is, \overline{s}} = \sum_{s=1}^n T^{s}_{is, \overline{j}} = - \eta_{i,\overline{j} } = 0.$$
On the other hand, by (\ref{eq:3.3}), we get
$$ 2 R_{i\overline{j}k\overline{\ell}, \,m}    = - T^{\ell}_{ik, \,\overline{j}m } - \overline{ T^k_{j\ell , \, \overline{i}\overline{m} }   }. $$
Interchange $i$ and $m$, and subtract the so obtained line from the above line, we get
\begin{eqnarray*}
2(R_{i\overline{j}k\overline{\ell}, \,m} - R_{m\overline{j}k\overline{\ell}, \,i} ) & = & \big(T^{\ell}_{mk,\,\overline{j}i} - T^{\ell}_{ik,\,\overline{j}m} \big) + \overline{\big( T^{k}_{j\ell ,\,\overline{m}\overline{i}} - T^{k}_{j\ell ,\,\overline{i}\overline{m}}   \big)} \\
& = & \big(T^{\ell}_{mk,\,\overline{j}i} - T^{\ell}_{ik,\,\overline{j}m} \big)  - \sum_{r=1}^n T^r_{im} \, \overline{T^k_{j\ell , \, \overline{r} } }
\end{eqnarray*}
where in the last equality we used (\ref{eq:C1}). By the second Bianchi identity (\ref{eq:B2}), the left hand side of the above long equality equals to
$$ 2 \sum_{r=1}^n T^r_{im} R_{r\overline{j}k\overline{\ell}} =  - \sum_{r=1}^n T^r_{im}\big( T^{\ell}_{rk,\,\overline{j}} +\overline{ T^k_{j\ell , \, \overline{r}}    } \big) ,$$
where the last equality is due to (\ref{eq:3.3}) and the assumption that $c=0$. Putting the two sides together, we get
\begin{equation} \label{eq:3.7}
T^{\ell}_{mk,\,\overline{j}i} - T^{\ell}_{ik,\,\overline{j}m}  =  - \sum_{r=1}^n T^r_{im} \, T^{\ell}_{rk , \, \overline{j}}
\end{equation}
Letting $m=j=s$ and sum over $s$, and use the fact that $\sum_s T^{\ell}_{is, \, \overline{s}} =0$,  we get
$$ - \sum_{s=1}^n T^{\ell}_{ik, \, \overline{s}s} = -  \sum_{r,s=1}^n T^r_{is} \, T^{\ell}_{rk , \, \overline{s}}.$$
Replace $\ell$ by $j$, we get the second equality in the lemma.
\end{proof}

\vspace{0.3cm}

\begin{lemma} Suppose  $(M^n,g)$ is a compact Hermitian manifold with vanishing real bisectional curvature. Then under any unitary local frame $e$, we have
\begin{equation} \label{eq:3.8}
  \sum_{s=1}^n  T^{j}_{ik, \,\overline{s}s} =   \sum_{s=1}^n  T^{j}_{ik, \, s\overline{s}}
\end{equation}
for any $i$, $j$, $k$.
\end{lemma}

\begin{proof}
By \cite[Theorem 1.4]{YZ}, we know that the first, second, and third Ricci curvature of $R$ are all zero. So the lemma is implied by equality (\ref{eq:C3}).
\end{proof}

\vspace{0.3cm}

\noindent {\bf Proof of Proposition 2.} \ \ Let $(M^3,g)$ be a balanced Hermitian threefold. Let $e$ be a local unitary tangent frame. Define a $3\times 3$ matrix valued function $A$ by
$$ A_{i\alpha } = T^{\alpha}_{jk}, \ \ \ \ 1\leq i, \alpha \leq 3$$
where $(ijk)$ is a cyclic permutation of $(123)$. Since $g$ is balanced, we have $\eta_i = \sum_{s} T^s_{si}=0$ for each $i$, so
$$ A_{ij } = T^{j}_{jk} = - T^i_{ik} = T^i_{ki} = A_{ji},$$
and similarly, $A_{ik} = A_{ki}$. Hence $A$ is symmetric. Consider the semi-positive definite Hermitian matrix $A\overline{A}=AA^{\ast}$, where $A^{\ast}$ stands for conjugate transpose. There exists $U(3)$-valued function $U$ such that $U(AA^{\ast})U^{\ast}=D$ is diagonal, with nonnegative entries on the diagonal line. But this matrix is $B\overline{B}$, where $B=UA\,^t\!U$. Write $B=B_1+\sqrt{-1}B_2$ for its real and imaginary parts, we have
$$ D = B \overline{B} = \big( B_1B_1 + B_2B_2 \big)  - \sqrt{-1} [B_1, B_2] $$
Therefore $[B_1, B_2]=0$ as $D$ is real. So $B_1$ and $B_2$ are commutative real symmetric matrices, hence can be simultaneously diagonalized by orthogonal matrices, namely, there exists $O(3)$-valued function $V$ such that $VB\,^t\!V$ is diagonal. Under the new unitary frame $\tilde{e} = (VU)e$, the corresponding matrix $A$ is diagonal, which means $T^i_{i\ell}=0$ for any $i$, $\ell$. This proves Proposition 2. \qed

\vspace{0.3cm}

Now we are ready to prove Theorem 1.

\vspace{0.3cm}

\noindent {\bf Proof of Theorem 1.} \ \ Let $(M^n,g)$ be a compact Hermitian manifold with vanishing real bisectional curvature. Consider the smooth function $f\geq 0$ on $M^n$ which at any $p\in M$ is defined by $\,f=\sum_{i,j,k=1}^n  |T^j_{ik}|^2$, where $e$ is any unitary frame of type $(1,0)$ tangent vectors in a neighborhood of $p$. Following Boothby \cite{Boothby}, we want to consider the operator $Lf=\sum_{s =1}^n f_{,s \overline{s}}$, which is independent of the choice of unitary frame thus is globally defined. Here as usual the indices after comma stands for covariant differentiation with respect to the Chern connection $\nabla$. We have
\begin{eqnarray}
\sum_{s} f_{,s \overline{s}} & = & \sum_{i,j,k,s} \big( T^{j}_{ik,\,s} \,\overline{ T^{j}_{ik} } + T^{j}_{ik} \,\overline{  T^{j}_{ik,\,\overline{s}}  } \big)_{,\,\overline{s}} \nonumber \\
& = & \sum_{i,j,k,s} \big( |T^{j}_{ik, \, s} |^2 +  |T^{j}_{ik, \, \overline{s}} |^2 \big) + P  \label{eq:3.9}
\end{eqnarray}
where
\begin{eqnarray}
 P & = & \sum_{i,j,k,s} \big(   T^{j}_{ik,\,s \overline{s} } \ \overline{ T^{j}_{ik} } + T^{j}_{ik} \,\overline{  T^{j}_{ik,\,\overline{s} s }  } \big)  \nonumber   \\
& = & 2 \,{\mathcal R}e \!\sum_{i,j,k,s,r} T^r_{is} \, T^j_{rk,\,\overline{s}} \, \overline{T^j_{ik } }  \nonumber \\
& = & 2 \,{\mathcal R}e \!\sum_{i,j,k,s,r} \big( T^r_{is} \, T^j_{rk}\big) _{,\overline{s}} \, \overline{T^j_{ik } } \label{eq:3.10}
\end{eqnarray}
where the second equality is due to (\ref{eq:3.8}) and the second equation of (\ref{eq:3.6}), while the last equality is due to the first equation of (\ref{eq:3.6}).

If one can show that the function $P$ is everywhere non-negative, then $\sum_s f_{,s\overline{s}} \geq 0$. So by Bochner's lemma (cf. \cite{Boothby}), the compactness of $M^n$ would imply that  $f$ must be constant, hence by (\ref{eq:3.9}) $T^j_{ik,\,\overline{s}} =0$ for any $i$, $j$, $k$, $s$. From (\ref{eq:3.3}), we get $R=0$, namely, the Hermitian manifold  $(M^n,g)$ is Chern flat.

For general $n$, we do not know how to prove $P\geq 0$ at this time, this is why we settle for the $n=3$ case in Theorem 1.  In this case, since $g$ is balanced by \cite[Theorem 1.4]{YZ}, we know by Proposition 2 that there always exists a local unitary frame $e$ so that $T^1_{1k} = T^2_{2k} = T^3_{3k}=0$ for any $k$. In other words, the only possibly non-zero components of $T$ are $T^j_{ik}$ where  $(ijk)$ is a permutation of $(123)$. Under such a frame, for fixed distinct $i$, $j$, $k$, we have
$$ \sum_{r=1}^3 T^r_{is} \, T^j_{rk} = T^i_{is} \, T^j_{ik} = 0, $$
so by (\ref{eq:3.10}) we obtain $P=0$, which leads to $\sum_s f_{,s\overline{s}} \geq 0$, therefore $f$ is constant and $R=0$. This completes the proof of Theorem 1. \qed

\vspace{0.3cm}

\noindent\textbf{Acknowledgments.} {The second named author would like to thank his former coauthor Xiaokui Yang, who is the driving force  in forming the notion of real bisectional curvature and studying its property and related questions. He would also like to thank mathematicians  Haojie Chen, Xiaolan Nie, Kai Tang, Bo Yang, and Quanting Zhao for helpful discussions.}

\end{document}